\documentclass{amsart}
\pdfoutput=1

\usepackage{microtype}
\usepackage[mathscr]{eucal}
\usepackage{mathtools}
\usepackage{amsthm}
\usepackage{enumitem}
\usepackage{hyperref}
\usepackage{mleftright}

\DeclareMathOperator{\vol}{vol}
\DeclareMathOperator{\ave}{ave}
\DeclareMathOperator{\area}{area}
\DeclareMathOperator{\scal}{Sc}
\DeclareMathOperator{\two}{II}
\DeclareMathOperator{\tr}{tr}
\DeclareMathOperator{\conv}{conv}
\DeclareMathOperator{\gr}{Gr}
\DeclareMathOperator{\End}{End}

\theoremstyle{plain}
\newtheorem{theorem}{Theorem}[section]

\newtheorem{lemma}[theorem]{Lemma}
\newtheorem{corollary}[theorem]{Corollary}
\theoremstyle{remark}
\newtheorem{remark}[theorem]{Remark}

\theoremstyle{definition}
\newtheorem{example}[theorem]{Example}

\begin{document}
	
\title[Normal curvature bounds]{Normal curvature bounds\\ for immersions into Riemannian domains}

\author{Matteo Raffaelli}
\address{School of Mathematics, Georgia Institute of Technology, Atlanta, Georgia 30332}
\email{raffaelli@math.gatech.edu}
\date{July 14, 2026}

\begin{abstract} 
We study Gromov's problem on the minimal normal curvature of immersions. Our main result is a lower bound for the average normal curvature of a closed submanifold immersed in a Riemannian domain. The bound is expressed in terms of an invariant measuring the optimal $n$-trace convexity of the domain under a unit-gradient normalization. As applications, we recover and extend Petrunin's lower bound for closed submanifolds immersed in Euclidean balls to geodesic balls in Cartan--Hadamard manifolds and, more generally, to Riemannian domains satisfying suitable convexity conditions. In the Cartan--Hadamard setting, under a natural assumption on the average scalar curvature, we show that equality forces the submanifold to lie minimally in the boundary sphere and that the radial sectional curvature vanishes along it. We also obtain sharper estimates for immersions into hyperbolic balls and Euclidean tubes.
\end{abstract}
\maketitle

\section{Introduction}

In \cite{petrunin2024}, Petrunin proved that the normal curvature of any closed flat submanifold $L^{n\geq 1}$ immersed in the closed unit ball $\mathbb B^{m}\subset \mathbb R^{m}$ is bounded below by $(3n/(n+2))^{1/2}$; here by \textit{normal curvature} we mean the maximum of $\lvert \two(u,u)\rvert$ over unit vectors $u$ in $T L$, where $\two$ denotes the second fundamental form~\cite[p.~3]{gromov2025}.

Petrunin’s argument relies on a new version of the Gauss equation. Gromov observed that this formula extends naturally to arbitrary Riemannian ambient spaces, but the Euclidean lower bound itself requires additional structure~\cite[sec.~14]{gromov2025}. We identify such structure as a quantitative form of $n$-trace convexity of the ambient domain. 
This yields, among other things, new Petrunin-type estimates for submanifolds immersed in Cartan--Hadamard balls, hyperbolic balls, and tubes.

To state our first result, let $\iota\colon L^n \hookrightarrow(M^{m}, g)$ be a closed (i.e., compact without boundary) Riemannian submanifold, let $\Omega\subset M$ be a subset such that $\iota(L)\subset \Omega$, and denote by $\kappa(p)$ the average of $\lvert \two(u,u)\rvert^2$ over unit vectors $u$ in $T_pL$. For any continuous function $\varphi$ on $L$, let $\ave(\varphi)$ denote its average:
\begin{equation*}
\ave(\varphi)= \frac{1}{\vol(L)}\int_L \varphi\, dV.
\end{equation*}

\begin{theorem}\label{thm:main}
Suppose there exists a smooth function $f$ on an open neighborhood of $\Omega$ in $M$ such that, on $\Omega$,
\begin{equation*}
\nabla^2 f \geq \lambda g,\quad \lvert \nabla f \rvert \leq \mu
\end{equation*} 
for some positive constants $\lambda,\mu$. Then
\begin{equation*}
\ave(\kappa) \geq \frac{3n}{n+2} \mleft( \frac{\lambda}{\mu}\mright)^2 + \frac{2}{n(n+2)}(\ave(\scal_n(TL) - \scal_L)) ,
\end{equation*}
where $\scal_n$ denotes the ambient $n$-scalar curvature and $\scal_L$ the scalar curvature of $L$.
\end{theorem}

\begin{remark}
If $\ave(\kappa)\geq a\geq 0$, then the normal curvature is bounded below by $\sqrt a$.
\end{remark}

Let $d$ denote the ambient Riemannian distance. Note that when $\Omega$ is the closed Euclidean ball of radius $r$ centered at $p\in\mathbb R^{m}$, the function $f(q) = d(p,q)^2/2$ satisfies the assumptions of the theorem with $\lambda = 1$ and $\mu =r$. By the Hessian comparison theorem, the same conclusion holds for any closed geodesic ball $B_r^m$ in a Cartan--Hadamard manifold; recall that a Cartan--Hadamard manifold is a complete simply connected Riemannian manifold with nonpositive sectional curvature.
Thus Theorem~\ref{thm:main} yields the following extension of \cite[Thm.~1.2(b)]{petrunin2024}.

\begin{corollary}\label{cor:1}
If $\Omega = B_r^m$ and $\ave(\scal_L)\leq \ave(\scal_n(TL))$, then
\begin{equation}\label{eq:cor}
\ave(\kappa) \geq \frac{3n}{r^2(n+2)}.
\end{equation}
\end{corollary}

In the Euclidean case, $\scal_n =0$, so the hypothesis becomes simply $\ave(\scal_L)\leq 0$, which already yields the standard $3n/(n+2)$ bound when $r=1$. In \cite[p.~4]{petrunin2024b}, Petrunin notes that immersions of $\mathbb T^n$ into $\mathbb B^{m}$ with flat induced metric and $\kappa = 3n/(n+2)$ are spherical. We extend this observation as follows.

\begin{corollary}\label{cor:2}
Suppose $\Omega = B_r^m$ and $\ave(\scal_L) \leq \ave(\scal_n(TL))$. If equality holds in \eqref{eq:cor}, then $L$ is a minimal submanifold of $\partial B_r^m$, and the ambient sectional curvature vanishes in every plane spanned by $\nabla\rho$ and a vector tangent to $L$. 
\end{corollary}

We emphasize that the Cartan--Hadamard assumption in Corollary~\ref{cor:1} is only needed to obtain the global estimate on arbitrary geodesic balls with the sharp constants $\lambda =1$ and $\mu =r$. Locally, no curvature sign assumption is required: in any Riemannian manifold, the squared distance function is strictly convex on sufficiently small geodesic balls centered at $p$. Hence every such ball satisfies the hypotheses of Theorem~\ref{thm:main}.

We will obtain Theorem~\ref{thm:main} as a consequence of a more general result in which the auxiliary convex function is replaced by a quantity depending only on the ambient domain. Let $\gr_n(TM)$ be the Grassmannian of $n$-planes tangent to $M$, and let $\Omega\subset M$ be a nonempty subset. Given a smooth function $f$ on an open neighborhood of $\Omega$ in $M$, define
\begin{equation*}
A_n(f,\Omega)= \inf_{\substack{q\in\Omega\\P\in \gr_n(T_q M)}}\tr_P (\nabla^2 f),
\end{equation*}
where 
\begin{equation*}
\tr_P(\nabla^2f)= \sum_{i=1}^n\nabla^2f(e_i,e_i)
\end{equation*}
for any orthonormal basis $(e_1,\dotsc, e_n)$ of $P$. Thus $A_n(f,\Omega)$ is the best uniform lower bound for the Hessian trace of $f$ over all $n$-planes in $TM\rvert_\Omega$, i.e., the sum of the $n$ smallest Hessian eigenvalues. Then set
\begin{equation*}
\Theta_n(\Omega)= \sup_{\sup_\Omega \lvert \nabla f\rvert \leq 1}  A_n(f,\Omega),
\end{equation*}
where the supremum is taken over all smooth functions $f$ defined on open neighborhoods of $\Omega$ in $M$ and satisfying $\sup \lvert \nabla f\rvert \leq 1$ on $\Omega$. This definition is inspired by the Harvey--Lawson theory of $p$-convexity~\cite{harvey2012, harvey2013}; see also \cite{sha1986, wu1987} and section~\ref{sec:theta}. In their terminology, functions whose Hessian has nonnegative trace on every $n$-plane are $n$-plurisubharmonic, and domains admitting $n$-plurisubharmonic exhaustions are $n$-convex. 
Informally, $\Theta_n(\Omega)$ measures the strongest $n$-trace convexity that $\Omega$ admits under the unit-gradient normalization. 

\begin{theorem}[The main result]\label{thm:main2}
\begin{equation}\label{eq:main2}
\ave(\kappa) \geq \frac{3}{n(n+2)} \Theta_n(\Omega)^2 + \frac{2}{n(n+2)}(\ave(\scal_n(TL) - \scal_L)).
\end{equation}
\end{theorem}

\begin{remark}
$\Theta_n(B_r^m) \geq n/r$, with equality in the Euclidean case; see Remark~\ref{rem:thetaeuclidean}.
\end{remark}

Theorem~\ref{thm:main2} has two advantages over Theorem~\ref{thm:main}. First, it applies in situations not covered by Theorem~\ref{thm:main}. For instance, if the ambient domain is compact without boundary, then no smooth function can satisfy the required Hessian bound, since any such function must attain a maximum. Second, even when the hypotheses of Theorem~\ref{thm:main} hold, Theorem~\ref{thm:main2} can give a sharper estimate, because the domain may be more convex in the $n$-trace sense than the Hessian lower bound can detect.

We first illustrate this second phenomenon for geodesic balls in hyperbolic space. In this setting, the invariant $\Theta_n(\Omega)$ coincides with the surface-to-volume ratio of the $n$-ball, as in the Euclidean case; see Remark~\ref{rem:thetahyperbolic}.

\begin{corollary}\label{cor:3}
If $\Omega = B_r^m \subset \mathbb H^{m}_{-c^2}$ and $\ave(\scal_L)\leq -c^2n(n-1)$, then
\begin{equation*}
\ave(\kappa) \geq \frac{3c^2}{n(n+2)} \mleft(\frac{\sinh^{n-1}(cr)}{\int_0^{cr}\sinh^{n-1}(s)\,ds}\mright)^2.
\end{equation*}
\end{corollary}

Clearly, as $c\to0$, the lower bound in Corollary~\ref{cor:3} converges to the sharp Euclidean constant $3n/((n+2)r^2)$. However, for every $c>0$, equality in the hyperbolic lower bound is not attained by any closed immersed submanifold; see Remark~\ref{rem:hyperbolic}.

As another concrete application of Theorem~\ref{thm:main2}, we give a new normal curvature obstruction for closed submanifolds immersed in Euclidean tubes, again expressed in terms of the average of $\kappa$. 

\begin{corollary}\label{cor:4}
Let $U^k \subset \mathbb R^k$, where $1\leq k<n$. If $\Omega = U^k \times B_r^{m-k} \subset\mathbb R^{m}$ and $\ave(\scal_L)\leq 0$, then
\begin{equation*}
\ave(\kappa) \geq \frac{3(n-k)^2}{n(n+2)r^2}.
\end{equation*}
\end{corollary}

Note that when $U$ is bounded and long in some direction, the estimate above is stronger than the one obtained from Theorem~\ref{thm:main}. Indeed, suppose that $U$ contains a line segment of length $2\ell$. Then any function $f$ satisfying the hypotheses of Theorem~\ref{thm:main} must have $\lambda/\mu \leq 1/\ell$. To see this, restrict $f$ to the segment and parametrize it by arclength. Since $(f\circ\gamma)'' \geq \lambda$ and $\lvert(f\circ\gamma)'\rvert \leq \mu$, the derivative must increase by at least $2\lambda\ell$ across the segment while it remains in $[-\mu,\mu]$, and so $2\lambda\ell \leq 2\mu$. Thus the convexity term in Theorem~\ref{thm:main} is bounded above by $3n/((n+2)\ell^2)$, which is smaller than the corresponding term in Corollary~\ref{cor:4} whenever $\ell >nr/(n-k)$.

Corollary~\ref{cor:4} remains valid if the Euclidean set $U^k$ is replaced by a closed Riemannian manifold, provided the hypothesis $\ave(\scal_L)\leq 0$ is replaced by $\ave(\scal_L)\leq \ave(\scal_n(TL))$. The proof carries over verbatim, since the estimate of the $n$-trace of the Hessian uses only the dimension of the first factor. As Remark~\ref{rem:tube} indicates, equality in the Euclidean tube estimate is not attained; nevertheless, Example~\ref{ex:sharp} shows that the same bound is sharp in this more general setting.

The proof of Theorem~\ref{thm:main2} is given in section~\ref{sec:proofs1}. We briefly outline the argument. For each admissible function $f$ in the definition of $\Theta_n(\Omega)$, the restriction of $f$ to $L$ has integral Laplacian equal to zero, because $L$ is closed. A standard submanifold Laplacian formula expresses $\Delta_Lf\rvert_L$ as the sum of the tangential trace of the ambient Hessian and the inner product $\langle H, \nabla f\rangle$, where $H$ is the mean curvature vector. The first term is bounded below by $A_n(f,\Omega)$, while the normalization $\sup_\Omega\lvert\nabla f\rvert\leq 1$ bounds the second term below by $-\lvert H\rvert$. After integrating over $L$ and optimizing over $f$, one obtains the estimate 
\begin{equation}\label{eq:averageH}
\ave(\lvert H\rvert) \geq \Theta_n(\Omega).
\end{equation}
The desired lower bound for $\ave(\kappa)$ then follows by averaging Petrunin's Gauss formula.

A useful by-product of the argument is the estimate~\eqref{eq:averageH} itself, which can be viewed as a convexity-based obstruction to minimality and low bending. By the triangle inequality, it implies that the average of $\lvert\two(u,u)\rvert$ over unit vectors in $TL$ is at least $\Theta_n(\Omega)/n$. When $n=1$ and $\Omega =\mathbb B^3$, these two estimates recover Fáry's classical theorem on the average curvature of a closed curve in $\mathbb B^3$~\cite{fary1950, tabachnikov2003}; see also \cite{chakerian1962, chakerian1964, veeravalli2001} for refinements and extensions. The Euclidean analogue of \eqref{eq:averageH} follows from \cite[Thm.~28.2.5]{burago1988}.
 
This work should be viewed as part of an emerging program, initiated by Gromov~\cite{gromov2022, gromov2022b, gromov2023, gromov2025}, on immersions with controlled normal curvature. A guiding question is to determine, for a closed smooth manifold, the least possible normal curvature among immersions into a Euclidean ball, and to understand the topological and geometric consequences of attaining the optimal bound. Related results include Petrunin’s optimality theorem for Gromov’s tori~\cite{petrunin2024} and his Veronese rigidity theorem~\cite{petrunin2024b}, as well as subsequent developments by Mendes~\cite{mendes2025} and Chodosh--Li~\cite{chodosh2026}.

\section{A submanifold Laplacian formula}
In this section we recall a basic submanifold formula for the Laplacian of the restriction of an ambient function.

Let $\iota\colon N^n \hookrightarrow (M^{m},g)$ be an immersed submanifold, with the induced metric. We will often simply write $N$ and $M$, omitting the dimensions from the notation. Let $f$ be a smooth function on $M$, and denote its gradient and Hessian by $\nabla f$ and $\nabla^2 f$, respectively. Then, for any pair of vector fields $V,W$ on $N$,
\begin{equation*}
\nabla^2_N f\rvert_N(V,W) = \nabla^2 f(V,W) + \langle \nabla f ,\two(V,W)\rangle,
\end{equation*}
where $f\rvert_N = f\circ\iota$ and, in the second term, we identify $\iota_\ast V$ and $\iota_\ast W$ with $V$ and $W$; see \cite[Prop.~1.2]{dajczer2019}.

Let $(e_1,\dotsc, e_n)$ be a local orthonormal frame on $N$, and let $\Delta_N$ be the Laplace operator on $N$. Recall that 
\begin{equation*}
\Delta_N f\rvert_N = \sum_{i=1}^n \nabla^2_N f\rvert_N(e_i,e_i).
\end{equation*}
Using the preceding identity, we obtain
\begin{equation*}
\Delta_N f\rvert_N = \sum_{i=1}^n\nabla^2 f(e_i,e_i) + \sum_{i=1}^n\langle \nabla f ,\two(e_i,e_i)\rangle.
\end{equation*}
The first term on the right-hand side is the tangential trace of the ambient Hessian, which we denote by $\tr_{TN} (\nabla^2 f)$. Writing $H= \sum_i \two(e_i,e_i)$ for the mean curvature vector, the formula becomes
\begin{equation}\label{eq:tr}
\Delta_N f\rvert_N = \tr_{TN} (\nabla^2 f)+ \langle H,  \nabla f \rangle.
\end{equation}

\section{Petrunin's Gauss formula}

To prove Theorem~\ref{thm:main2}, we will need the Riemannian analogue of Petrunin’s Gauss formula, which we review in this section.

In \cite{petrunin2024}, Petrunin showed that for any submanifold $N^n\subset \mathbb R^{m}$,
\begin{equation*}
\scal_N = \frac{3}{2} \lvert H\rvert^2 - \frac{n(n+2)}{2}\kappa.
\end{equation*}
As observed by Gromov, this formula has a direct analogue in the Riemannian setting. For an $n$-plane $P\subset T_pM$, let $\scal_n(P)$ denote the scalar curvature obtained by tracing the ambient curvature tensor over $P$:
\begin{equation*}
\scal_n(P) = \sum_{i,j=1}^n \langle R(e_i,e_j)e_j,e_i\rangle,
\end{equation*}
where $(e_1,\dotsc, e_n)$ is any orthonormal basis of $P$. Thus, for a submanifold $N^n\subset M^m$, $\scal_n(TN)$ denotes the function $p\mapsto \scal_n(T_pN)$.

\begin{lemma}[\cite{gromov2025}]\label{lm:gromov}
For any immersed submanifold $N^n\subset M^{m}$,
\begin{equation}\label{eq:gromov}
\scal_N = \scal_n(TN) + \frac{3}{2} \lvert H\rvert^2 - \frac{n(n+2)}{2}\kappa.
\end{equation}
\end{lemma}

\begin{proof}
Combining Petrunin’s spherical averaging identity
\begin{equation*}
\kappa = \frac{\lvert H\rvert^2 + 2\lvert \two\rvert^2}{n(n+2)}
\end{equation*}
with the traced Gauss equation
\begin{equation*}
\scal_N = \scal_n(TN) + \lvert H\rvert^2 - \lvert \two\rvert^2,
\end{equation*}
we obtain \eqref{eq:gromov}.
\end{proof}

\section{Proof of Theorem~\ref{thm:main2}}\label{sec:proofs1}
We are now ready to prove the two theorems stated in the introduction. We begin with Theorem~\ref{thm:main2}.

\begin{proof}[Proof of Theorem~\textup{\ref{thm:main2}}]
Let $f$ be smooth on an open neighborhood of $\Omega$ in $M$ and satisfy $\sup_\Omega \lvert \nabla f\rvert \leq 1$. Since $L$ is closed, the divergence theorem gives
\begin{equation*}
\int_L \Delta_L f\rvert_L\,dV =0.
\end{equation*}
Thus, integrating \eqref{eq:tr}, we obtain
\begin{equation*}
0 = \int_L \tr_{TL} (\nabla^2 f) \,dV + \int_L \langle H,  \nabla f \rangle\, dV.
\end{equation*}
Since $T_pL$ is an $n$-plane in $T_pM$, we have
\begin{equation}\label{eq:tr2}
\tr_{TL} (\nabla^2 f) \geq A_n(f,\Omega).
\end{equation}
Moreover, by Cauchy--Schwarz,
\begin{equation}\label{eq:Hu}
\langle H, \nabla f\rangle \geq - \lvert\langle H, \nabla f\rangle \rvert \geq - \lvert H\rvert \lvert \nabla f\rvert \geq -\lvert H\rvert.
\end{equation}
Substituting these estimates into the integral identity yields
\begin{equation*}
0 \geq A_n(f,\Omega) \vol(L) -  \int_L\lvert H\rvert\, dV.
\end{equation*}
Equivalently,
\begin{equation*}
\ave(\lvert H\rvert)= \frac{1}{\vol(L)}\int_L\lvert H\rvert\, dV\geq A_n(f,\Omega).
\end{equation*}
This holds for every admissible $f$, so taking the supremum over all $f$ with $\sup_\Omega\lvert\nabla f\rvert \leq 1$ gives
\begin{equation*}
\ave(\lvert H\rvert)\geq \Theta_n(\Omega).
\end{equation*}
Note that $\Theta_n(\Omega)\geq 0$, since constant functions are admissible in the definition of $\Theta_n$ and satisfy $A_n(f,\Omega) =0$. Hence
\begin{equation}\label{eq:Hbound}
\ave(\lvert H\rvert)^2\geq \Theta_n(\Omega)^2.
\end{equation}

Next we average \eqref{eq:gromov}: 
\begin{equation*}
\ave(\scal_L) =  \ave(\scal_n(TL)) + \frac{3}{2} \ave( \lvert H\rvert^2 ) - \frac{n(n+2)}{2}\ave(\kappa).
\end{equation*}
Rearranging,
\begin{equation*}
\ave(\kappa) = \frac{3}{n(n+2)}  \ave( \lvert H\rvert^2 ) + \frac{2}{n(n+2)} (\ave(\scal_n(TL)) - \ave(\scal_L) ).
\end{equation*}
By Cauchy--Schwarz,
\begin{equation*}
\ave(\lvert H\rvert)^2 = \frac{1}{\vol(L)^2} \mleft( \int_L \lvert H\rvert \,dV\mright)^2 \leq \frac{1}{\vol(L)^2} \int_L \lvert H\rvert^2\,dV \int_L 1\,dV = \ave( \lvert H\rvert^2 ).
\end{equation*}
Hence,
\begin{equation*}
\ave(\kappa) \geq \frac{3}{n(n+2)}  \ave(\lvert H\rvert)^2 + \frac{2}{n(n+2)} (\ave(\scal_n(TL)) - \ave(\scal_L) ).
\end{equation*}
Substituting \eqref{eq:Hbound}, we obtain \eqref{eq:main2}, as desired.
\end{proof}

We now derive Theorem~\ref{thm:main} from Theorem~\ref{thm:main2}. Note that it suffices to consider the case $\mu=1$. Indeed, for $\widetilde f = f/\mu$ we have $\lvert\nabla \widetilde f\rvert\leq 1$ and $\nabla^2\widetilde f\geq (\lambda/\mu) g$, so the general case follows from the normalized one.

\begin{proof}[Proof of Theorem~\textup{\ref{thm:main}}]
Suppose there exists a smooth function $f$ on an open neighborhood of $\Omega$ in $M$ with $\nabla^2f\geq \lambda g$ and $\lvert\nabla f\rvert\leq 1$ on $\Omega$. 
For every $q\in \Omega$ and every $n$-plane $P\in \gr_n(T_qM)$, the Hessian lower bound gives
\begin{equation*}
\tr_P (\nabla^2 f) = \sum_{i=1}^n\nabla^2 f(e_i,e_i) \geq \sum_{i=1}^n \lambda g(e_i,e_i)= n\lambda,
\end{equation*}
where $(e_1,\dotsc, e_n)$ is any orthonormal basis  of $P$.
Taking the infimum over all such $P$, we obtain
\begin{equation*}
A_n(f,\Omega) \geq n\lambda.
\end{equation*}
Since $\Theta_n(\Omega) \geq A_n(f,\Omega)$, it follows that
\begin{equation*}
\Theta_n(\Omega) \geq n\lambda.
\end{equation*}
Substituting this bound into Theorem~\ref{thm:main2} completes the proof. 
\end{proof}

\section{Immersions into Cartan--Hadamard balls}

Here we prove Corollaries~\ref{cor:1} and \ref{cor:2}, beginning with the first.

\begin{proof}[Proof of Corollary~\textup{\ref{cor:1}}]
Suppose $(M,g)$ is a Cartan--Hadamard manifold. Fix $p\in M$, and define
\begin{equation*}
\rho(q) = d(p,q),\quad f(q) = \frac{1}{2}\rho(q)^2.
\end{equation*}
Since the exponential map $\exp_p\colon T_p M\to M$ is a global diffeomorphism, the distance function $\rho$ is smooth on $M\setminus \{p\}$. Using
\begin{equation*}
\nabla^2 f = d\rho \otimes d\rho + \rho\,\nabla^2\rho
\end{equation*}
together with the Hessian comparison theorem~\cite[Thm.~6.7]{dajczer2019},
\begin{equation*}
\nabla^2\rho \geq \frac{1}{\rho}(g - d\rho \otimes d\rho),
\end{equation*}
we obtain $\nabla^2f\geq g$ on $M\setminus\{p\}$. In fact, $\nabla^2f =g$ at $p$, so the same inequality holds on all of $M$; see \cite[Ex.~5.9.27]{petersen2016}.
Moreover, since $\lvert \nabla\rho \rvert =1$, we have
\begin{equation*}
\lvert \nabla f \rvert = \rho \lvert\nabla\rho\rvert = \rho.
\end{equation*}
Therefore, on every closed geodesic ball $B_r^m$ centered at $p$, the function $f$ satisfies 
\begin{equation*}
\nabla^2 f \geq \lambda g,\quad \lvert \nabla f \rvert \leq \mu
\end{equation*} 
with $\lambda =1$ and $\mu=r$. Applying Theorem~\ref{thm:main}, the result follows.
\end{proof}

We then turn to the rigidity statement in Corollary~\ref{cor:2}.

\begin{proof}[Proof of Corollary~\textup{\ref{cor:2}}]
Let $\rho = d(p,\cdot)$, where $p$ is the center of $B_r^m$, and choose $f=\rho^2/(2r)$. Note that $f$ is an admissible test function in the definition of $\Theta_n(B_r^m)$, because $\lvert \nabla f\rvert =\rho/r \leq 1$ away from $p$ and $\nabla f(p)=0$. Moreover, since $\nabla^2 f\geq g/r$ by the Hessian comparison theorem, we have
\begin{equation*}
\tr_P(\nabla^2 f) \geq \frac{n}{r}
\end{equation*}
for every $n$-plane $P$.
At the center $p$, one has $\nabla^2 f= g/r$, so the lower bound is attained there. Hence
\begin{equation*}
A_n(f,B_r^m)= \frac{n}{r},\quad \Theta_n(B_r^m) \geq  \frac{n}{r}.
\end{equation*}

Let $D= \ave(\scal_n(TL) - \scal_L)\geq 0$. Recall that by the proof of Theorem~\ref{thm:main2},
\begin{align*}
\ave(\kappa) &= \frac{3}{n(n+2)} \ave(\lvert H\rvert^2) + \frac{2}{n(n+2)} D\\
&\geq \frac{3}{n(n+2)} \ave(\lvert H\rvert)^2 + \frac{2}{n(n+2)} D \\
&\geq \frac{3n}{r^2(n+2)} + \frac{2}{n(n+2)} D  \\
&\geq \frac{3n}{r^2(n+2)}.
\end{align*}
Suppose equality holds in \eqref{eq:cor}. Then equality must hold at every step in this chain of inequalities. Hence
\begin{equation*}
D=0,\quad \ave(\lvert H\rvert) =\frac{n}{r}, \quad \ave(\lvert H\rvert^2)=\ave(\lvert H\rvert)^2.
\end{equation*}
Therefore $\lvert H\rvert$ is constant and $\lvert H\rvert =n/r$.
Subtracting
\begin{equation*}
0 = \frac{n}{r}\vol(L) - \int_L \lvert H\rvert\, dV
\end{equation*}
from the integral identity
\begin{equation*}
0 = \int_L \tr_{TL} (\nabla^2 f) \,dV + \int_L \langle H,  \nabla f \rangle\, dV,
\end{equation*}
we obtain
\begin{equation*}
0 = \int_L \mleft(\tr_{TL} (\nabla^2 f)-\frac{n}{r}\mright) \,dV + \int_L (\langle H,  \nabla f \rangle +\lvert H\rvert )\, dV.
\end{equation*}
By \eqref{eq:tr2} and \eqref{eq:Hu}, both integrands are pointwise nonnegative. Since their sum has integral zero, both must vanish identically. Therefore,
\begin{equation*}
\tr_{TL} (\nabla^2 f)=\frac{n}{r}
\end{equation*}
and
\begin{equation*}
\langle H,  \nabla f \rangle = -\lvert H\rvert. 
\end{equation*}
The second equality gives equality in \eqref{eq:Hu}, namely
\begin{equation*}
\langle H, \nabla f\rangle =- \lvert H\rvert \lvert \nabla f\rvert= - \lvert H\rvert.
\end{equation*}
Since $\lvert H\rvert = n/r >0$, it follows that $\lvert \nabla f\rvert =1$ on $L$. But $\lvert \nabla f\rvert = \rho/r$, and so $\rho =r$ along $L$. Thus, $L\subset \partial B_r^m$. 

We next show that $L$ is minimal in $\partial B_r^m$. Since $f=\rho^2/(2r)$, on $M\setminus\{p\}$ we have
\begin{equation*}
\nabla^2 f = \frac{1}{r} d\rho \otimes d\rho + \frac{1}{r}\rho\,\nabla^2\rho,
\end{equation*}
and we have already shown that $\rho =r$ along $L$. Thus, for every $V\in TL$, 
\begin{equation*}
\nabla^2 f(V,V) = \nabla^2\rho(V,V),
\end{equation*}
which implies
\begin{equation*}
\tr_{TL} (\nabla^2 \rho)= \tr_{TL} (\nabla^2 f)=\frac{n}{r}.
\end{equation*}

Let $\nu = \nabla \rho$ denote the outward unit normal to $\partial B_r^m$, and let $\two^{\partial B_r^m \subset M}$ be the second fundamental form of $\partial B_r^m$. For tangent vector fields $V,W$ on $L$, viewed also as tangent to $\partial B_r^m$, we have
\begin{equation*}
\two^{\partial B_r^m \subset M} (V,W)= \langle \nabla_V W,\nu\rangle\nu = -\nabla^2\rho(V,W)\nu,
\end{equation*}
where the second equality follows by differentiating $\langle W,\nu\rangle =0$ and using $\nu=\nabla\rho$. 
Therefore, tracing over an orthonormal frame $(e_1,\dotsc, e_n)$ for $TL$, we obtain
\begin{equation}\label{eq:nu}
\sum_{i=1}^n\two^{\partial B_r^m \subset M} (e_i,e_i)= -\frac{n}{r}\nu.
\end{equation}
On the other hand, the second fundamental form of $L\subset M$ decomposes as
\begin{equation*}
\two = \two^{L\subset\partial B_r^m} + \two^{\partial B_r^m \subset M}\lvert_{TL},
\end{equation*}
where $\two^{L\subset\partial B_r^m}$ is the second fundamental form of $L$ as a submanifold of $\partial B_r^m$. Taking traces along $L$, this gives
\begin{equation*}
H = H^{L\subset\partial B_r^m} + \sum_{i=1}^n\two^{\partial B_r^m \subset M}(e_i,e_i).
\end{equation*}
Finally, the equality $\langle H,\nabla f\rangle = - \lvert H\rvert \lvert\nabla f\rvert$ implies that $H$ and $\nabla f$ point in opposite directions. Since $\nabla f = (\rho/r)\nabla \rho$ and $\rho =r$ on $L$, we have
\begin{equation*}
H = -\lvert H\rvert \frac{\nabla f}{\lvert \nabla f\rvert} = -\frac{n}{r} \nu.
\end{equation*}
Combining this with \eqref{eq:nu}, we conclude that $H^{L\subset\partial B_r^m}=0$. Thus $L$ is minimal in $\partial B_r^m$, as claimed.

It remains only to check that the radial sectional curvature vanishes along $L$. Since $\nabla^2 f \geq g/r$, and since $\nabla^2 f=\nabla^2\rho$ on vectors tangent to $L$, we have $\nabla^2\rho (e_i,e_i) \geq 1/r$. On the other hand, $\tr_{TL}(\nabla^2\rho) =n/r$, and so $\nabla^2\rho (e_i,e_i) = 1/r$. Hence, choosing the frame so that one of its vectors is parallel to $v$, bilinearity gives
\begin{equation}\label{eq:rigidity}
\nabla^2\rho(v,v) = \frac{1}{r}\lvert v\rvert^2\quad \text{for all }v\in T_qL.
\end{equation}

Let $\gamma\colon[0,r]\to M$ be the radial geodesic from $p$ to $q$, with
\begin{equation*}
\gamma(r) = q,\quad \gamma'(r)=\nabla\rho(q).
\end{equation*}
For $t>0$, let $S_t\colon \gamma'(t)^\perp\to \gamma'(t)^\perp$ be the shape operator of the geodesic sphere of radius $t$, and set
\begin{equation*}
Z_t= S_t - \frac{1}{t}I_t,
\end{equation*}
where $I_t$ is the identity on $\gamma'(t)^\perp$. Since $\langle S_t \cdot, \cdot\rangle = \nabla^2\rho(\cdot,\cdot)$, equation \eqref{eq:rigidity} gives $\langle Z_r v,v\rangle =0$ for all $v\in T_qL$. As $\langle Z_r \cdot, \cdot \rangle \geq 0$ by Hessian comparison, it follows that $Z_r v=0$ for all $v\in T_qL$. 

Now let $V(t)$ be the parallel translate of $v$ along $\gamma$, and set
\begin{equation*}
z(t) = \langle Z_t V(t), V(t)\rangle.
\end{equation*}
Since $z(t)\geq 0$ and $Z_r v=0$, the function $z$ has a minimum at $t=r$, which gives 
\begin{equation*}
z_-'(r)= \lim_{t\to r^-} \frac{z(t)-z(r)}{t-r}\leq 0.
\end{equation*}
On the other hand, the shape operator $S_t$ satisfies the Riccati equation
\begin{equation*}
D_t S_t +S_t^2 +R(\cdot,\gamma'(t))\gamma'(t) =0,
\end{equation*}
where $R$ is the Riemann curvature tensor and $D_t$ denotes the induced connection on $\End(\gamma'(t)^\perp)$; see \cite[Cor.~3.2.10]{petersen2016}. Since $D_tI_t=0$, substituting $S_t = t^{-1} I_t +Z_t$, we get
\begin{equation*}
D_tZ_t +Z_t^2 +\frac{2}{t}Z_t + R(\cdot,\gamma'(t))\gamma'(t)=0.
\end{equation*}
Evaluating at $t=r$ and using $Z_r(v) =0$, we find
\begin{equation*}
z_{-}'(r) = -\langle R(v,\nabla\rho)\nabla\rho, v \rangle \geq 0,
\end{equation*}
because $M$ has nonpositive sectional curvature. Hence $z'_-(r) =0$, and the claim follows.
\end{proof}

\begin{remark}\label{rem:thetaeuclidean}
In the proof above we obtained $\Theta_n(B_r^m) \geq n/r$. We now show that equality holds in the Euclidean case:
\begin{equation*}
\Theta_n(B_r^m \subset \mathbb R^{m}) = \frac{n}{r}.
\end{equation*}

Let $f$ be smooth on a neighborhood $U$ of $B_r^m\subset \mathbb R^{m}$ and satisfy $\lvert\nabla f\rvert \leq 1$ on $B_r^m$. Fix an $n$-dimensional Euclidean plane $P$ through the center $p$, and set 
\begin{equation*}
B_r^n = P\cap B_r^m.
\end{equation*}
Since $P\cap U$ is totally geodesic, hence minimal, equation~\eqref{eq:tr} gives
\begin{equation*}
\Delta_P (f\rvert_{P\cap U}) = \tr_P (\nabla^2 f) 
\end{equation*}
on $P\cap U$. By the divergence theorem on $B_r^n$,
\begin{equation*}
\int_{B_r^n}\tr_P (\nabla^2 f)\, dV = \int_{\partial B_r^n} \langle \nabla^P(f\rvert_{P\cap U}),\nu \rangle\, dS,
\end{equation*}
where $\nu$ is the outward unit normal to $\partial B_r^n$ in $P$. Since $\lvert \nabla^P(f\rvert_{P\cap U})\rvert \leq \lvert\nabla f\rvert\leq 1$, we have
\begin{equation*}
\int_{B_r^n}\tr_P (\nabla^2 f)\, dV \leq \area(\partial B_r^n).
\end{equation*}
Dividing by $\vol(B_r^n)$, and using 
\begin{equation*}
\frac{\area(\partial B_r^n)}{\vol(B_r^n)} = \frac{n}{r},
\end{equation*}
we obtain
\begin{equation*}
\frac{1}{\vol(B_r^n)}\int_{B_r^n}\tr_P (\nabla^2 f)\, dV \leq \frac{n}{r}.
\end{equation*}
Hence there exists a point $q\in B_r^n$ such that 
\begin{equation*}
\tr_P (\nabla^2 f)(q) \leq\frac{n}{r}.
\end{equation*}
Since $A_n(f,B_r^m)$ is the infimum of $\tr_Q (\nabla^2 f)$ over all points $q \in B_r^m$ and all $n$-planes $Q\subset T_q\mathbb R^{m}$, it follows that
\begin{equation*}
A_n(f,B_r^m) \leq\frac{n}{r}.
\end{equation*}
This holds for every admissible $f$, so $\Theta_n(B_r^m\subset \mathbb R^{m}) \leq n/r$, which proves the claim.

We finally note that the same argument works beyond balls. Indeed, let $\Omega\subset\mathbb R^m$, and let $P\subset \mathbb R^m$ be an affine $n$-plane for which $\Omega\cap P$ is a compact domain with smooth boundary in $P$. Then
\begin{equation*}
\Theta_n(\Omega) \leq \frac{\area(\partial (\Omega\cap P) )}{\vol(\Omega\cap P)}.
\end{equation*}
\end{remark}

\section{Immersions into hyperbolic balls}\label{sec:hyperbolic}

In this section we prove Corollary~\ref{cor:3}.

\begin{proof}[Proof of Corollary~\textup{\ref{cor:3}}]
Let $\Omega = B_r^m\subset \mathbb H_{-c^2}^{m}$, and let $\rho = d(p,\cdot)$, where $p$ is the center of the ball. Set
\begin{equation*}
\sigma(t)=\sinh^{n-1}(ct),\quad C = \frac{\sigma(r)}{\int_0^r \sigma(t)\,dt}.
\end{equation*}
We will prove that $\Theta_n(B_r^m) \geq C$. Define $\phi\colon [0,r+\varepsilon] \to\mathbb R$, for some small $\varepsilon>0$, by
\begin{equation*}
\phi'(t) = C\,\frac{\int_0^t \sigma(s)\, ds}{\sigma(t)},\quad \phi(0)=0.
\end{equation*}
Near $t=0$, the expression above extends as an odd analytic function, so $\phi$ is even and analytic near $0$. Since $\rho^2$ is smooth near the center $p$, it follows that $f = \phi\circ\rho$ extends smoothly across $p$.

We next verify the gradient bound. Set
\begin{equation*}
F(t)=\frac{\int_0^t \sigma(s)\,ds}{\sigma(t)}.
\end{equation*}
Since $\lvert \nabla f\rvert =\phi'(\rho)$, $\phi'(t) =CF(t)$, and $\phi'(r)=1$, it is enough to show that $F$ is increasing on $[0,r]$. For $n=1$, this is immediate. For $n\geq 2$, the function $\sigma'(t)/\sigma(t) = (n-1)c\coth(ct)$ is decreasing. Hence, for $0<s\leq t$,
\begin{equation*}
\frac{\sigma'(s)}{\sigma(s)} \geq \frac{\sigma'(t)}{\sigma(t)}.
\end{equation*}
Multiplying by $\sigma(s)$ and integrating from $0$ to $t$, we get 
\begin{equation*}
\sigma(t) = \int_0^t \sigma'(s)\, ds \geq \frac{\sigma'(t)}{\sigma(t)}\int_0^t \sigma(s)\, ds.
\end{equation*}
Therefore $F'(t)\geq 0$, as desired.

It remains to estimate the $n$-trace of $\nabla^2 f$. For $q\neq p$, set $t=\rho(q)$. Since
\begin{equation*}
\nabla^2\rho = c\coth(c\rho)(g- d\rho\otimes d\rho)
\end{equation*}
in hyperbolic space, the Hessian chain rule gives 
\begin{equation*}
\nabla^2f =\phi''(t)\,d\rho\otimes d\rho +\phi'(t)c\coth(ct)(g- d\rho\otimes d\rho).
\end{equation*}
Set $h(t) = c\coth(ct)\phi'(t)$. Differentiating the definition of $\phi'(t)$ yields
\begin{equation*}
\phi''(t) + (n-1)h(t) =C.
\end{equation*}
We next show that $h(t) \geq C/n$. Since, for $0\leq s\leq t$, 
\begin{equation*}
\frac{d}{ds} \sinh^n (cs) = nc \sinh^{n-1}(cs)\cosh (cs) \leq nc \cosh (ct) \sinh^{n-1} (cs),
\end{equation*}
integration gives
\begin{equation}\label{eq:sinh}
\sinh^n (ct) \leq nc \cosh(ct)\int_0^t\sigma(s)\,ds.
\end{equation}
Dividing by $n\sinh^n (ct)$, we obtain
\begin{equation*}
c\coth(ct) \frac{\int_0^t \sigma(s)\, ds}{\sigma(t)} \geq \frac{1}{n}.
\end{equation*}
Thus $h(t) \geq C/n$, and consequently 
\begin{equation*}
\phi''(t) = C- (n-1)h(t) \leq h(t).
\end{equation*}

Now let $P \subset T_q H_{-c^2}^{m}$ be any $n$-plane, and let $(e_1,\dotsc, e_n)$ be an orthonormal basis of $P$. Set
\begin{equation*}
a = \sum_{i=1}^n d\rho(e_i)^2.
\end{equation*}
Since $a$ is the squared length of the projection of $\nabla\rho$ onto $P$, we have $0\leq a\leq 1$. Therefore
\begin{align*}
\tr_P(\nabla^2 f) &= a\phi''(t) + (n-a)h(t) \\
&= C + (1-a)(h(t)-\phi''(t)) \\
&\geq C.
\end{align*}
By continuity, the same inequality holds at $p$. Hence $A_n(f,B_r^m) \geq C$, and so $\Theta_n(B_r^m) \geq C$. Applying Theorem~\ref{thm:main2}, and using the fact that, in $\mathbb H_{-c^2}^{m}$, 
\begin{equation*}
\scal_n(TL) = -c^2n(n-1),
\end{equation*}
the corollary follows.
\end{proof}

\begin{remark}\label{rem:thetahyperbolic}
The constant $C$ is exactly the surface-to-volume ratio of a hyperbolic $n$-ball. The preceding proof gives $\Theta_n(B_r^m) \geq C$. The reverse inequality follows as in Remark~\ref{rem:thetaeuclidean}, replacing the Euclidean $n$-plane through $p$ by a totally geodesic copy of $\mathbb H_{-c^2}^n$. Hence 
\begin{equation*}
\Theta_n(B_r^m\subset \mathbb H_{-c^2}^{m}) =C.
\end{equation*}
\end{remark}

\begin{remark}\label{rem:hyperbolic}
In contrast to the Euclidean case, equality in Corollary~\ref{cor:3} is not attained by any closed submanifold. Indeed, since $\Theta_n(B_r^m) \geq C$ and $\ave(\lvert H\rvert) \geq \Theta_n(B_r^m)$, we have 
\begin{equation*}
\ave(\lvert H\rvert)  \geq C.
\end{equation*}
Suppose that equality holds in Corollary~\ref{cor:3}. Then, under the scalar curvature assumption, the same argument used in the proof of Corollary~\ref{cor:2} yields
\begin{equation*}
\lvert H\rvert = C>0, \quad \tr_{TL} (\nabla^2 f) =C,\quad \langle H,\nabla f\rangle = -\lvert H\rvert,
\end{equation*}
where $f=\phi\circ\rho$. The last equality implies $\lvert \nabla f\rvert =1$ on $L$. However, $\lvert \nabla f\rvert =\phi'(\rho)\leq 1$, with equality only when $\rho=r$, since $F$ is strictly increasing. Hence 
\begin{equation*}
L\subset \partial B_r^m.
\end{equation*}
If $n=m$, this is already impossible for dimensional reasons. Thus assume $n<m$. If $P=T_qL\subset T_q\partial B_r^m$, then $a=0$, so
\begin{equation*}
C=\tr_{TL} (\nabla^2 f) = nc\coth(cr).
\end{equation*}
This is impossible, since inequality~\eqref{eq:sinh} is strict at $t=r$, giving $C<nc\coth(cr)$.
\end{remark}

\section{Immersions into tubes}

Finally, we prove the tube estimate of Corollary~\ref{cor:4}.

\begin{proof}[Proof of Corollary~\textup{\ref{cor:4}}]
For $x\in U^k$ and $y\in\mathbb B^{m-k}_r$, define
\begin{equation*}
f(x,y) = \frac{\lvert y\rvert^2}{2r}.
\end{equation*}
Then $\lvert \nabla f\rvert = \lvert y\rvert/r\leq 1$ and
\begin{equation*}
\nabla^2f = \frac{1}{r} \pi^\ast_y g,
\end{equation*}
where $\pi_y$ denotes projection onto the $\mathbb R^{m-k}$-factor. Let $P \subset T_{(x,y)}\mathbb R^{m}$ be any $n$-plane, and let $(e_1,\dotsc, e_n)$ be an orthonormal basis of $P$. Then
\begin{equation*}
\tr_P(\nabla^2f) = \frac{1}{r} \sum_{i=1}^n \lvert \pi_y e_i \rvert^2.
\end{equation*}
Since $\lvert e_i\rvert^2 = \lvert \pi_x e_i\rvert^2 + \lvert \pi_y e_i\rvert^2 =1$, we get
\begin{equation*}
\sum_{i=1}^n\lvert \pi_y e_i\rvert^2 = n- \sum_{i=1}^n\lvert \pi_x e_i\rvert^2 
\end{equation*}
But the trace of the $x$-projection satisfies
\begin{equation*}
\sum_{i=1}^n\lvert \pi_x e_i\rvert^2\leq k;
\end{equation*}
to see this, let $(b_1,\dotsc, b_k)$ be an orthonormal basis of $T_xU$, and note that
\begin{equation*}
\sum_{i=1}^n\lvert \pi_x e_i\rvert^2 = \sum_{i=1}^n \sum_{\alpha=1}^k \langle e_i,b_\alpha\rangle^2 =\sum_{\alpha=1}^k \sum_{i=1}^n \langle e_i,b_\alpha\rangle^2 = \sum_{\alpha=1}^k \lvert \pi_P b_\alpha\rvert^2.
\end{equation*}
Therefore
\begin{equation*}
\tr_P(\nabla^2 f) \geq \frac{n-k}{r}.
\end{equation*}
Finally, taking the infimum over all $n$-planes $P$, we obtain 
\begin{equation*}
\Theta_n(\Omega)\geq A_n(f,\Omega) \geq\frac{n-k}{r},
\end{equation*}
and the desired inequality follows from Theorem~\ref{thm:main2}.
\end{proof}

\begin{remark}\label{rem:tube}
The estimate in Corollary~\ref{cor:4} is not attained. To see this, suppose equality holds for a closed submanifold $L$, and use the test function from the proof. 
Then equality must hold in the trace estimate
\begin{equation*}
\tr_{T_pL}(\nabla^2 f) \geq \frac{n-k}{r}.
\end{equation*} 
If $(e_1,\dotsc,e_n)$ is an orthonormal basis of $T_pL$, equality means
\begin{equation*}
\sum_{i=1}^n \lvert \pi_x e_i\rvert^2 =k.
\end{equation*}
Since the left-hand side is the trace of the orthogonal projection of $T_pL$ onto the $k$-dimensional horizontal factor, the whole horizontal subspace must be contained in $T_pL$. Thus 
\begin{equation*}
\mathbb R^k \times \{0\} \subset T_pL\quad\forall p\in L.
\end{equation*}
This is impossible for a closed submanifold. Indeed, let $a\in \mathbb R^k\times \{0\}$ be nonzero. Then the function $\varphi(p)=  \langle \iota(p),a\rangle$ has intrinsic gradient equal to $a$ everywhere, because $a$ is tangent to $L$ at every point. Hence $\lvert\nabla^L \varphi\rvert = \lvert a\rvert >0$, contradicting the fact that $\varphi$ must attain a maximum on $L$.
\end{remark}

\begin{example}\label{ex:sharp}
Let 
\begin{equation*}
N^{n-k}\subset \partial B_r^{m-k}\subset \mathbb R^{m-k}
\end{equation*}
be a closed flat submanifold that is minimal in the Euclidean sphere $\partial B_r^{m-k}$. Let $F^k$ be any closed Riemannian manifold, and consider the product immersion
\begin{equation*}
L= F^k \times  N^{n-k} \subset F^k\times \mathbb R^{m-k}.
\end{equation*}
Since $TL = TF\oplus TN$ and the Euclidean factor contributes no ambient curvature, 
\begin{equation*}
\scal_n(TL) = \scal_F.
\end{equation*}
On the other hand, 
\begin{equation*}
\scal_L = \scal_F +\scal_{N} =\scal_F,
\end{equation*}
because $N$ is flat. Hence $\scal_n(TL) = \scal_L$. 
Moreover, since $N$ is minimal in $\partial B_r^{m-k}$, its Euclidean mean curvature vector has constant length $(n-k)/r$. The same is true for the mean curvature vector of $L$, since directions tangent to $F$ contribute nothing to the second fundamental form. Applying \eqref{eq:gromov}, we obtain 
\begin{equation*}
\kappa=\frac{3(n-k)^2}{n(n+2)r^2}.
\end{equation*}
Thus the estimate in Corollary~\ref{cor:4} is sharp for Riemannian product domains. 
\end{example}

\section{The invariant $\Theta_n$}\label{sec:theta}

In this section, we examine the quantity $\Theta_n$ more closely. Recall that, by definition,
\begin{equation*}
\Theta_n(\Omega)= \sup_{\sup_\Omega \lvert \nabla f\rvert \leq 1} \inf_{\substack{q\in\Omega\\P\in \gr_n(T_q M)}}\tr_P (\nabla^2 f),
\end{equation*}
with $f$ smooth on some open neighborhood of $\Omega$ in $M$. Equivalently, if 
\begin{equation*}
\lambda_1^f (q)\leq \dotsb\leq \lambda_m^f (q)
\end{equation*}
are the eigenvalues of $\nabla^2 f(q)$, then
\begin{equation*}
\Theta_n(\Omega)= \sup_{\sup_\Omega \lvert \nabla f\rvert \leq 1} \inf_{q\in\Omega} \sum_{i=1}^n \lambda_i^f(q).
\end{equation*}
In the Euclidean case, the operator $\sum_{i=1}^n \lambda_i^f$ is known as the truncated Laplacian; see, e.g., \cite{birindelli2018, birindelli2021, bessa2022}.

We begin by recording the most basic properties of $\Theta_n$.

\begin{lemma}\label{lm:properties}
The following properties hold.
\begin{enumerate}[font=\upshape, label=(\roman*)]
\item \label{item:nonnegativity} \emph{Nonnegativity}. 
\begin{equation*}
\Theta_n(\Omega) \geq 0.
\end{equation*}
\item \label{item:inclusion}\emph{Monotonicity under inclusion}. If $\Omega_1\subset \Omega_2$, then 
\begin{equation*}
\Theta_n(\Omega_1) \geq \Theta_n(\Omega_2).
\end{equation*}
\item \label{item:isometry}\emph{Invariance under isometries}. If $\Phi \colon (M,g)\to (M',g')$ is an isometry, then 
\begin{equation*}
\Theta'_n(\Phi(\Omega)) = \Theta_n(\Omega).
\end{equation*}
\item \emph{Homothety covariance}. If $\Phi \colon (M,g)\to (M',g')$ is a homothety of factor $c >0$, i.e., 
\begin{equation*}
\Phi^\ast g' = c^2 g,
\end{equation*}
then
\begin{equation*}
\Theta'_n(\Phi(\Omega)) = \frac{1}{c}\Theta_n(\Omega).
\end{equation*}
\item\label{item:nmonotonicity} \emph{Normalized monotonicity in $n$}. If $1\leq n \leq k \leq m$, then
\begin{equation*}
\frac{\Theta_n(\Omega)}{n} \leq \frac{\Theta_k(\Omega)}{k}.
\end{equation*}
\end{enumerate}
\end{lemma}

\begin{remark}
Properties \ref{item:nonnegativity} and \ref{item:nmonotonicity} imply monotonicity in $n$:
\begin{equation*}
\Theta_1(\Omega) \leq \dotsb\leq \Theta_m(\Omega).
\end{equation*}
\end{remark}

\begin{proof}[Proof of Lemma~\textup{\ref{lm:properties}}]
\leavevmode
\begin{enumerate}[font=\upshape, label=(\roman*)]
\item Constant functions are admissible, since $\lvert \nabla f\rvert =0 \leq 1$. For such a function, $\nabla^2 f =0$, and therefore $A_n(f,\Omega)=0$. Taking the supremum over all admissible $f$, we get $\Theta_n(\Omega)\geq 0$.

\item Clearly, if $f$ is admissible for $\Omega_2$, then it is also admissible for $\Omega_1\subset\Omega_2$. Moreover, since the infimum over the smaller set is taken over fewer points and planes,
\begin{equation*}
\Theta_n(\Omega_1)\geq A_n(f,\Omega_1)\geq A_n(f,\Omega_2)
\end{equation*}
for every admissible $f$ on $\Omega_2$. Taking the supremum over all such $f$, we obtain $\Theta_n(\Omega_1) \geq \Theta_n(\Omega_2)$. 

\item Let $\Phi\colon (M, g) \to (M', g')$ be an isometry, and let $f'$ be smooth near $\Phi(\Omega)$. Set $f = f'\circ\Phi$. Since $\Phi$ is an isometry,
\begin{equation*}
\lvert \nabla f\rvert_g = \lvert \nabla' f'\rvert_{g'} \circ \Phi,
\end{equation*}
so $f$ is admissible for $\Omega$ if and only if $f'$ is admissible for $\Phi(\Omega)$. Moreover, for $v,w\in T_qM$,
\begin{equation*}
\nabla^2 f(v,w) = \nabla'^2 f'(d\Phi(v), d\Phi(w)).
\end{equation*}
Thus, if $P\subset T_qM$ is an $n$-plane, then so is $d\Phi(P)\subset T_{\Phi(q)} M'$ and
\begin{equation*}
\tr_P (\nabla^2 f) = \tr_{d\Phi(P)} (\nabla'^2 f').
\end{equation*}
Therefore $A_n(f,\Omega) = A_n'(f', \Phi(\Omega))$. Taking the supremum over all admissible $f'$ gives the equality in \ref{item:isometry}.

\item Let $f'$ be smooth near $\Phi(\Omega)$, and set $f= f'\circ\Phi$. Then
\begin{equation*}
\lvert \nabla f\rvert_g = c \lvert \nabla f'\rvert_{g'} \circ\Phi,
\end{equation*}
and so $f'$ is admissible for $\Phi(\Omega)$ if and only if $\widetilde f= f/c$ is admissible for $\Omega$. Since
\begin{equation*}
 \tr_P (\nabla^2 \widetilde f) = c \tr_{d\Phi(P)} \nabla'^2 f',
\end{equation*}
we obtain
\begin{equation*}
A_n(\widetilde f,\Omega) = c A'_n(f',\Phi(\Omega)).
\end{equation*}
Taking suprema over admissible $f'$ gives
\begin{equation*}
\Theta_n(\Omega) \geq c \Theta'_n(\Phi(\Omega)).
\end{equation*}
Applying the same argument to $\Phi^{-1}$, whose homothety factor is $1/c$, gives the reverse inequality. 

\item Fix an admissible function $f$. If $1\leq n \leq k \leq m$, then
\begin{equation*}
\frac{1}{n}\sum_{i=1}^n\lambda_i^f(q)\leq \frac{1}{k}\sum_{i=1}^k \lambda_i^f(q).
\end{equation*}
Taking the infimum over $q\in\Omega$, we get
\begin{equation*}
\frac{1}{n}A_n(f,\Omega) \leq \frac{1}{k}A_k(f,\Omega).
\end{equation*}
Finally, taking the supremum over all admissible $f$ gives the desired property.
\end{enumerate}
\end{proof}

We next discuss two limiting cases that help place $\Theta_n$ in a broader context. First suppose that $n=m$. Then the $n$-trace of the Hessian is simply the Laplacian:
\begin{equation*}
\Theta_m(\Omega) = \sup_{\sup_\Omega\lvert \nabla f\rvert\leq 1} \inf_\Omega \Delta f.
\end{equation*}
Thus $\Theta_m$ is the gradient-field analogue of the Bessa--Montenegro divergence constant $c(\Omega)$~\cite{bessa2003}. More precisely, it is obtained by restricting the vector fields in the definition of $c(\Omega)$ to those of the form $X=\nabla f$. Consequently,
\begin{equation*}
\Theta_m(\Omega) \leq c(\Omega)\leq h(\Omega),
\end{equation*}
where $h(\Omega)$ denotes the Cheeger constant; see \cite{cheeger1970} and \cite[Rem.~2.7]{bessa2003}.

At the opposite end, when $n=1$, we have the following lemma.

\begin{lemma}\label{lm:circumradius}
If $\Omega\subset \mathbb R^m$ is any compact convex set, then
\begin{equation*}
\Theta_1(\Omega) = \frac{1}{R(\Omega)},
\end{equation*}
where $R(\Omega)$ denotes the circumradius.
\end{lemma}

\begin{remark}
If $R(\Omega)=0$, then $\Omega$ is a point and $\Theta_1(\Omega) = +\infty$, so the formula holds with the convention $1/0 = +\infty$. We therefore assume $R(\Omega)>0$ in the proof.
\end{remark}

\begin{proof}[Proof of Lemma~\textup{\ref{lm:circumradius}}]
By Lemma~\ref{lm:properties}\ref{item:isometry}, we may apply a rigid motion and assume that the origin $o$ is a circumcenter of $\Omega$. Since $\Omega \subset B^m_{R(\Omega)}$ and $\Theta_1(B^m_{R(\Omega)}) = 1/R(\Omega)$, where $B^m_{R(\Omega)}$ is the ball of radius $R(\Omega)$ centered at $o$, Lemma~\ref{lm:properties}\ref{item:inclusion} gives
\begin{equation*}
\Theta_1(\Omega) \geq \frac{1}{R(\Omega)}.
\end{equation*}

It remains to prove the reverse inequality $\Theta_1(\Omega) \leq 1/R(\Omega)$. Take any admissible $f$, so $\lvert \nabla f\rvert \leq 1$ on $\Omega$, and suppose that 
\begin{equation*}
A_1(f, \Omega) > a>0.
\end{equation*} 
Then
\begin{equation*}
\nabla^2 f (v,v) \geq a \lvert v\rvert^2 \quad \text{for all }q\in \Omega \text{ and } v\in T_q\mathbb R^m.
\end{equation*} 
We will show that this forces $a \leq 1/R(\Omega)$. Once this is established, the desired conclusion follows: since $a$ can be chosen arbitrarily close to $A_1(f,\Omega)$, we have $A_1(f,\Omega)\leq 1/R(\Omega)$; taking the supremum over all admissible $f$, we conclude that $\Theta_1(\Omega) \leq 1/R(\Omega)$.

Let $R= R(\Omega)$. Note that $o \in \conv(\Omega \cap \partial B^m_R)$, because $B^m_R$ is the minimal enclosing ball of $\Omega$. Then, by Carathéodory's theorem, there are $N\leq m+1$ contact points
\begin{equation*}
x_1,\dotsc, x_N \in \Omega \cap \partial B^m_R
\end{equation*}
and weights $\lambda_i \geq 0, \sum_i\lambda_i =1$, such that
\begin{equation*}
o = \sum_i\lambda_i x_i.
\end{equation*}
Moreover, since $\nabla^2 f(\cdot, \cdot) \geq a \lvert\cdot\rvert^2$, the gradient of $f$ increases at least at rate $a$ along every direction. Applying this along the segment from $o$ to $x_i$, we get
\begin{equation*}
\langle \nabla f(x_i) - \nabla f(o), x_i\rangle \geq a \lvert x_i\rvert^2.
\end{equation*}
Since $\lvert x_i\rvert = R$, this becomes
\begin{equation*}
\langle \nabla f(x_i) - \nabla f(o), x_i\rangle \geq a R^2.
\end{equation*}
Multiplying by $\lambda_i$ and summing over $i$ gives
\begin{equation*}
\sum_i\lambda_i\langle \nabla f(x_i), x_i\rangle - \sum_i \lambda_i \langle \nabla f(o), x_i\rangle\geq a R^2.
\end{equation*}
The second term vanishes, because $\sum_i\lambda_i x_i =o$. Therefore
\begin{equation*}
\sum_i\lambda_i\langle \nabla f(x_i), x_i\rangle \geq a R^2.
\end{equation*}
On the other hand, since $\lvert \nabla f\rvert \leq 1$, we have
\begin{equation*}
\langle \nabla f(x_i), x_i\rangle \leq \lvert \nabla f(x_i)\rvert \lvert x_i\rvert \leq R.
\end{equation*}
Taking the weighted average gives
\begin{equation*}
\sum_i\lambda_i \langle \nabla f(x_i), x_i\rangle \leq R.
\end{equation*}
Hence
\begin{equation*}
a R^2\leq \sum_i\lambda_i \langle \nabla f(x_i), x_i\rangle \leq R,
\end{equation*}
and the desired inequality $a\leq 1/R$ follows.
\end{proof}

We conclude this section by computing the value of $\Theta_n$ for geodesic balls in the round sphere. This complements the Euclidean and hyperbolic ball computations already discussed in Remarks~\ref{rem:thetaeuclidean} and \ref{rem:thetahyperbolic}. The spherical case is analogous in spirit, but the optimal radial test function depends on the trace dimension.

\begin{lemma}
Let $B_r^m\subset \mathbb S^m$ be the closed geodesic ball of radius $r$ in the unit round sphere. If $0 < r<\pi$, then
\begin{align*}
\Theta_{n}(B_r^m) = 
\begin{cases}
\max\{ n\cot r, 0\}, \quad &1\leq n<m, \\[1pt]
\dfrac{\area(\partial B_r^m)}{\vol (B_r^m)} = \dfrac{\sin^{m-1} r}{\int_0^r \sin^{m-1} t\,dt},\quad &n=m.
\end{cases}
\end{align*}
If $r\geq \pi$, then $\Theta_{n}(B_r^m) =0$ for all $1 \leq n\leq m$.
\end{lemma}
\begin{proof}
If $r\geq \pi$, then $B_r^m = \mathbb S^m$. For any admissible function $f$, we have $\nabla^2 f\leq 0$ at a maximum point, and hence $A_n(f,\mathbb S^m)\leq 0$. By nonnegativity of $\Theta_n$, this gives $\Theta_n(\mathbb S^m)=0$. Thus we may assume $0<r<\pi$.

Suppose first that $n=m$. The argument is analogous to the hyperbolic case, with $\sin$ in place of $\sinh$. More precisely, one uses the radial test function determined by the spherical volume density $\sin^{m-1}t$, and obtains the lower bound 
\begin{equation*}
\Theta_m(B_r^m)\geq \frac{\area(\partial B_r^m)}{\vol (B_r^m)}.
\end{equation*}
The reverse inequality follows from the divergence theorem applied to an arbitrary admissible function on $B_r^m$. We omit the details.

It remains to consider the case $1\leq n<m$. First let $0< r<\pi/2$, and define
\begin{equation*}
f = -\frac{\cos\rho}{\sin r}.
\end{equation*}
Since $\nabla^2 (\cos\rho) = -\cos(\rho) g$ on $\mathbb S^m$, we have
\begin{equation*}
\nabla^2 f = \frac{\cos\rho}{\sin r} g.
\end{equation*}
Also,
\begin{equation*}
\lvert \nabla f\rvert = \frac{\sin\rho}{\sin r} \leq 1
\end{equation*}
on $B_r^m$. Hence $f$ is admissible, and for every $n$-plane $P$,
\begin{equation*}
\tr_P(\nabla^2 f) = n\frac{\cos\rho}{\sin r}\geq n\frac{\cos r}{\sin r} = n\cot r.
\end{equation*}
Therefore
\begin{equation*}
\Theta_n(B_r^m)\geq n\cot r.
\end{equation*}

For the reverse inequality, let $f$ be admissible. Since $n<m$, we may choose an $n$-plane $P\subset T_q\partial B_r^m$ at a point $q\in \partial B_r^m$ where $f\rvert_{\partial B_r^m}$ attains its maximum. At $q$,
\begin{equation*}
\tr_P \bigl( \nabla^2_{\partial B_r^m} f\rvert_{\partial B_r^m} \bigr)\leq 0.
\end{equation*}
Using the submanifold Hessian formula and
\begin{equation*}
\two_{}(v,w) = -\cot r\langle  v,w\rangle \nabla\rho,
\end{equation*}
we get
\begin{equation*}
\tr_P(\nabla^2 f) = \tr_P \bigl( \nabla^2_{\partial B_r^m} f\rvert_{\partial B_r^m} \bigr) + n \cot r \langle \nabla f, \nabla\rho\rangle\leq n \cot r.
\end{equation*}
Thus $A_n(f, B_r^m) \leq n\cot r$, and taking the supremum over admissible $f$ gives $\Theta_n(B_r^m)\leq n\cot r$. Hence, for $0<r<\pi/2$,
\begin{equation*}
\Theta_n(B_r^m)= n\cot r.
\end{equation*}

Finally, let $\pi/2 \leq r<\pi$. If $r= \pi/2$, then the same boundary argument gives $\Theta_n(B_{\pi/2}^m)\leq 0$, and hence $\Theta_n(B_{\pi/2}^m)= 0$ by nonnegativity. Since $B_{\pi/2}^m \subset B_r^m$, monotonicity under inclusion and nonnegativity imply $\Theta_n(B_r^m) = 0$. This completes the proof. 
\end{proof}

\section*{Acknowledgment}
The author thanks Ricardo Mendes for many helpful comments, and in particular for suggesting Lemma~\ref{lm:circumradius}.

\bibliographystyle{amsplain}
\bibliography{references}
\end{document}